\documentclass[12pt]{article}
\usepackage{color}
\usepackage{graphicx}
\usepackage{colortbl}
\usepackage{array}
\usepackage{amssymb}
\usepackage{amsmath}
\usepackage{amsthm}
\usepackage{mathrsfs}
\usepackage{dcolumn}
\usepackage{longtable}
\usepackage{hhline}
\numberwithin{equation}{section}
\allowdisplaybreaks

\newtheorem{thm}{Theorem}[section]
\newtheorem{defn}[thm]{Definition}

\title{Logarithmic Generalization of the Lambert $W$ function and its Applications to Adiabatic Thermostatics of the Three-Parameter Entropy}

\author{\textbf{Cristina B. Corcino$^{1,2}$}\\ {\large\bf Roberto B. Corcino$^{1,2}$}\\ {$^1$Research Institute for Computational}\\ {Mathematics and Physics}\\{$^2$Department of Mathematics}\\Cebu Normal University\\Cebu City, Philippines \vspace{9pt} 
}

\begin{document}

\maketitle

\begin{abstract}
A generalization of the Lambert W function called the logarithmic Lambert function is found to be a solution to the  thermostatics of the three-parameter entropy of classical ideal gas in adiabatic ensembles. The derivative, integral, Taylor series, approximation formula and branches of the function are obtained. The thermostatics are computed and the heat functions are expressed in terms of the logarithmic Lambert function. 


%\bigskip
%\noindent {\bf PACS numbers}. 02.10.-v, 02.70.Rr, 05.90.+m 

\end{abstract}

\bigskip

\section{Introduction}
In thermodynamics, entropy is a measure of the number of specific ways in which a thermodynamic system may be arranged, commonly understood as a measure of disorder. According to the second law of thermodynamics the entropy of an isolated system never decreases; such a system will spontaneously proceed towards thermodynamic equilibrium, the configuration with maximum entropy \cite{Gibbs}.

\smallskip
Within thermodynamics, thermostatics is the physical theory that deals with the equilibrium states, and with transformations where time s not an eplicit variable; it ignores the flows, i.e. the time derivatives of quantities such as the energy or the number of particles (See \cite{Thermostatics}). 

\smallskip
A system in thermodynamic equilibrium with its surroundings can be described using three macroscopic variables corresponding to the thermal, mechanical, and the chemical equilibrium. For each fixed value of these macroscopic variables (macrostates) there are many possible microscopic configurations (microstates). A collection of systems existing in the various possible microstates, but characterized by the same macroscopic variables is called an ensemble. The adiabatic class has the heat function as its thermal equilibrium variable. The specific form of each of the four adiabatic ensembles, its heat function and corresponding entropy are listed below (see \cite{Chandrashekar-Segar}).

\bigskip
\begin{table*}[htbp]
	\centering
		\begin{tabular}{|c|c|c|}
		\hline
			\textbf{Ensemble} & \textbf{Heat Function} 			  & \textbf{Entropy}\\
			\hline
			Microcanonical        			& Internal Energy     & \\
			$(N,V,E)$             			& $E$                 &$S(N,V,E)$\\
			\hline
			Isoenthalpic-isobaric 			& Enthalpy        	  & \\
			$(N,P,H)$             			&$H=E+PV$       			&$S(N,P,H)$\\
			\hline
			Third Adiabatic ensemble 		&Hill Energy					&\\
			$(\mu,V,L)$									&$L=E-\mu N$					&$S(\mu,V,L)$\\	
			\hline
			Fourth adiabatic ensemble   &Ray Energy						&\\
			$(\mu,P,R)$									&$R=E+PV-\mu N$				&$S(\mu,P,R)$\\
			\hline
			
		\end{tabular}
	\caption{Adiabatic ensembles}
	\label{tab:Table1AdiabaticEnsembles}
\end{table*}

\bigskip
It is known that some physical systems cannot be described by Boltzmann-Gibbs(BG) statistical mechanics (\cite{Asgarani},\cite{151569}). Among these physical systems are diffusion \cite{Shlesinger-Zaslavsky-Klafter}, turbulence \cite{Beck}, transverse momentum distribution of hadron jets in $e^+e^-$ collisions \cite{Bediaga-Curado-deMiranda}, thermalization of heavy quarks in collisional process \cite{Walton-Rafelski}, astrophysics \cite{Binney-Tremaine} and solar neutrines \cite{Clayton}. To overcome some difficulties in dealing with these systems, Tsallis \cite{Tsallis1988} introduced a generalized entropic form, the $q$-entropy
\begin{equation}
S_q=k\sum_{i=1}^{\omega}p_i\ln_q\frac{1}{p_i},
\end{equation}
where $k$ is a positive constant and $\omega$ is the total number of microscopic states. For $q>0$, $\ln_q x$ called the $q$-logarithm is defined by
\begin{equation}
\ln_q x=\frac{x^{1-q}-1}{1-q},\;\;\;\ln_1x=\ln x.
\end{equation}
The inverse function of the $q$-logarithm is called $q$-exponential and is given by
\begin{equation}
\exp_q x=[1+(1-q)x]^{\frac{1}{1-q}},\;\;\;\exp_1x=\exp x.
\end{equation}
In the case of equiprobability, BG is recovered in the limit $q\to1$. 

\smallskip
A two-parameter entropy $S_{q,q'}$ that recovered the $q$-entropy $S_q$ in the limit $q'\to1$ was defined in \cite{Schwammle-Tsallis} as
\begin{equation}
S_{q,q'}\equiv\sum_{i=1}^{\omega}p_i\ln_{q,q'}\frac{1}{p_i}=\frac{1}{1-q'}\sum_{i=1}^{\omega}p_i\left[\exp\left(\frac{1-q'}{1-q}(p_i^{q-1}-1)\right)-1\right].
\end{equation}
This entropy is useful for solving optimization problems \cite{Asgarani}. Applications of $S_q$ to a class of energy based ensembles were done in \cite{Chandrashekar-Mohammed} while applications of $S_{q,q'}$ to adiabatic ensembles were done in \cite{Chandrashekar-Segar}. Results in the applications of $S_{q,q'}$ involved the well-known Lambert W function.

\smallskip
A three-parameter entropy $S_{q,q',r}$ that recovers $Sq,q'$ in the limit $r\to1$ was defined in \cite{Corcino-Corcino} as
\begin{equation}\label{three-parameter entropy}
S_{q,q^{\prime },r} \equiv k\sum_{i=1}^w{p_i\ln_{ q,q^{\prime },r} \frac{1}{p_i}},
\end{equation}
where k is a positive constant and 
\begin{equation}
\label{lnq}
\ln_{q,q',r}x \equiv \frac{1}{1-r}\left(\exp\left(\frac{1-r}{1-q'}\left(e^{(1-q')\ln_q x}-1\right)-1\right)\right)
\end{equation}
The three-parameter entropic function \eqref{three-parameter entropy} was shown to be analytic (hence, Lesche-stable), concave and convex in specified ranges of the parameters (see \cite{Corcino-Corcino}). 

\smallskip
In this paper it is shown that adiabatic thermostatics of the three-parameter entropy of classical ideal gas involved a generalized Lambert W function. This generalized Lambert W function is studied and basic analytic properties are proved. The results are presented in Section 2. Its applications to the adiabatic thermostatics of the three-parameter entropy of classical ideal gas are derived in Section 3. Finally, a conclusion is given in Section 4.

\section{Logarithmic Generalization of the Lambert $W$ Function}

A generalization of the Lambert $W$ function will be called the logarithmic Lambert function denoted by $W_\mathcal{L}(x)$. Its formal definition is given below and fundamental properties of this function are proved. 

\begin{defn}\label{def1}\rm
For any real number $x$ and constant $B$, the logarithmic Lambert function $W_\mathcal{L}(x)$ is defined to be the solution to the equation
\begin{equation}
y\ln(By)e^y=x.
\label{defn of log lambert}
\end{equation}
\end{defn}

\smallskip
Observe that $y$ cannot be zero. Moreover, $By$ must be positive. By Definition \ref{def1}, $y=W_\mathcal{L}(x)$. The derivatives of $W_\mathcal{L}(x)$ with respect to $x$ can be readily determined as the following theorem shows. 

\begin{thm} The derivative of the logarithmic Lambert function is given by
\begin{equation}
\frac{dW_\mathcal{L}(x)}{dx}=\frac{e^{-W_\mathcal{L}(x)}}{[W_\mathcal{L}(x)+1]\ln BW_\mathcal{L}(x) +1}.
\label{derivatives of log lambert}
\end{equation}
\end{thm}
\begin{proof} Taking the derivative of both sides of \eqref{defn of log lambert} gives 
\begin{equation*}
\ln(By) \ ye^y \ \frac{dy}{dx}+\left(\ln(By)+1\right)e^y \ \frac{dy}{dx}=1, 
\end{equation*}
from which
\begin{equation}
\frac{dy}{dx}=\frac{1}{\left[y\ln(By)+\ln(By)+1\right]e^y}.
\label{dy/dx}
\end{equation}
With $y=W_\mathcal{L}(x)$, \eqref{dy/dx} reduces to \eqref{derivatives of log lambert}.
\end{proof}

\smallskip
The integral of the logarithmic Lambert function is given in the next theorem.
\begin{thm} The integral of $W_\mathcal{L}(x)$ is
\begin{align}
\int W_\mathcal{L}(x)\ dx &=e^{W_\mathcal{L}(x)}\left[1+\left(W^2_\mathcal{L}(x)-W_\mathcal{L}(x)+1\right)\ln\left(BW_\mathcal{L}(x)\right)\right]\nonumber\\
&\;\;\;\;-2Ei\left(W_\mathcal{L}(x)\right)+C,
\label{integral of log lambert}
\end{align}
where $Ei(x)$ is the exponential integral given by
$$Ei(x)=\int \frac{e^x}{x}dx.$$
\end{thm}
\begin{proof} From \eqref{defn of log lambert},
\begin{equation*}
dx=\left(y\ln(By)+\ln(By)+1\right)e^y\ dy.
\end{equation*}
Thus,
\begin{align}
\int y\ dx&=\int y\left(y\ln(By)+\ln(By)+1\right)e^y\ dy \nonumber \\
&=\int y^2e^y\ln(By)\ dy+\int ye^y\ln(By)\ dy+\int ye^y\ dy.
\label{integral of y dx}
\end{align}
These integrals can be computed using integration by parts to obtain
\begin{equation}
\int ye^y\ dy=(y-1)e^y+C_1,
\label{third integral}
\end{equation}
\begin{equation}
\int ye^y\ln(By)\ dy=e^y\left((y-1)\ln(By)-1\right)+Ei(y)+C_2,
\label{second integral}
\end{equation}
\begin{equation}
\int y^2e^y\ln(By)\ dy=e^y\left[(y^2-2y+2)\ln(By)-y+3\right]-2Ei(y)+C_3,
\label{first integral}
\end{equation}
where $C_1, C_2, C_3$ are constants. Substitution of \eqref{third integral}, \eqref{first integral} and \eqref{second integral} to \eqref{integral of y dx} with $C=C_1+C_2+C_3$, and writing $W_\mathcal{L}(x)$ for $y$  will give \eqref{integral of log lambert}.
\end{proof}

\smallskip
The next theorem contains the Taylor series expansion of $W_\mathcal{L}(x)$.

\begin{thm} Few terms of the Taylor series of $W_\mathcal{L}(x)$ about 0 are given below:
\begin{equation}
W_\mathcal{L}(x)=\frac{1}{B}+e^{-\frac{1}{B}}x+\frac{2+B}{2!}e^{-\frac{2}{B}}x^2+\frac{(4B^2+9B+9)e^{-\frac{3}{B}}}{3!}x^3+\cdots
\label{Taylor series}
\end{equation}
\end{thm}
\begin{proof} Being the inverse of the function defined by $x=y\ln(By)e^y$, the Lagrange inversion theorem is the key to obtain the Taylor series of the function $W_\mathcal{L}(x)$. \\
\indent Let $f(y)=y\ln(By)e^y$. The function $f$ is analytic for $By>0$. Moreover, $f'(y)=\left[(y+1)\ln(By)+1\right]e^y$, $f'(\frac{1}{B})=e^{\frac{1}{B}}\neq 0$, and for finite $B$, $f(\frac{1}{B})=0$. By the Lagrange Inversion Theorem (taking $a=\frac{1}{B}$),
\begin{equation}
W_\mathcal{L}(x)=\frac{1}{B}+\sum_{n = 1}^\infty g_n \frac{x^n}{n!},
\label{from Lagrange theorem}
\end{equation}
where
\begin{equation}
g_n=\lim \frac{d^{n-1}}{dy^{n-1}} \left(\frac{y-\frac{1}{B}}{f(y)}\right)^n.
\label{gn}
\end{equation}
The values of $g_n$ for $n=1,2,3$ are
\begin{align*}
g_1&=e^{-\frac{1}{B}} \\ g_2&=(2+B)e^{-\frac{2}{B}} \\ g_3&=(4B^2+9B+9)e^{-\frac{3}{B}}.
\end{align*}
Substituting these values to \eqref{from Lagrange theorem} will yield \eqref{Taylor series}.
\end{proof}

\smallskip
An approximation formula for $W_\mathcal{L}(x)$ expressed in terms of the  classical Lambert $W$ function is proved in the next theorem.

\begin{thm} For large $x$,
\begin{equation}
W_\mathcal{L}(x)\sim W(x)-\ln\left(\ln\left(BW(x)\right)\right),
\label{approximation}
\end{equation}
where $W(x)$ denotes the Lambert $W$ function.
\end{thm}
\begin{proof} From \eqref{defn of log lambert}, $y=W_\mathcal{L}(x)$ satisfies 
\begin{equation*}
x=y(\ln By)e^y\sim ye^y.
\end{equation*}
Then
\begin{equation}
y=W(x)+u(x),
\label{expression with u(x)}
\end{equation}
where $u(x)$ is a function to be determined. Substituting \eqref{expression with u(x)} to \eqref{defn of log lambert} yields
\begin{equation}
W(x)\left[1+\frac{u(x)}{W(x)}\right]\ln\left(BW(x)\left[1+\frac{u(x)}{W(x)}\right]\right)e^{W(x)}\cdot e^{u(x)}=x.
\label{expression1 for x}
\end{equation}
With $u(x)<<W(x)$, \eqref{expression1 for x} becomes 
\begin{equation}
W(x)e^{W(x)}\ln\left(BW(x)\right)e^{u(x)}=x.
\label{expression2 for x}
\end{equation}
By defition of $W(x)$, $W(x)e^{W(x)}=x$. Hence \eqref{expression2 for x} gives
\begin{equation}
\ln(BW(x))e^{u(x)}=1,
\label{identity}
\end{equation}
from which 
\begin{equation*}
u(x)=-\ln\left(\ln(BW(x))\right).
\end{equation*}
Thus,
\begin{equation*}
W_\mathcal{L}(x) \sim W(x)-\ln\left(\ln(BW(x))\right).
\end{equation*}
\end{proof}

\smallskip
The table below illustrates the accuracy of the approximation formula in \eqref{approximation}.

\begin{table*}[htbp]
	\centering
		\begin{tabular}{|r|c|c|c|}
		\hline
		$x$ & $W_\mathcal{L}(x)$	&  Approximate       & Relative Error\\
		    &                     &  Value             &\\
		\hline
	$302.7564$	& $4$ & $3.8914$ & $2.71438\times 10^{-2}$\\
		\hline
	$1194.3088$	& $5$ & $4.8766$ & $2.46807\times 10^{-2}$\\
		\hline
	$4337.0842$	& $6$ & $5.8756$ & $2.07321\times 10^{-2}$\\
	  \hline
	 $14937.6471$ & $7$ & $6.8792$ & $1.72518\times 10^{-2}$\\
	  \hline
	 $49589.8229$ & $8$ & $7.8844$ & $1.44500\times 10^{-2}$\\
	  \hline
	 $160238.6564$ & $9$ & $8.8899$ & $1.22306\times 10^{-2}$ \\
	  \hline
	 $507178.1179$ & $10$ & $9.8953$ & $ 1.04662\times 10^{-2}$\\
	 \hline
		\end{tabular}
\end{table*}
 
\bigskip
The next theorem describes the branches of the logarithmic Lambert function.  

\bigskip
\begin{thm}
Let $x=f(y)=y\ln(By)e^y$. Then the branches of the logarithmic Lambert function $y=W_\mathcal{L}(x)$ can be described as follows:
\begin{enumerate}
\item When $B>0$, the branches are
\begin{itemize}
\item $W^0_\mathcal{L}(x) : (f(\delta),+\infty)\to [\delta,+\infty)$ is strictly increasing;
\item $W^1_\mathcal{L}(x) : (f(\delta),0)\to [0,\delta)$ is strictly decreasing,
\end{itemize}
where $\delta$ is the unique solution to
\begin{equation}\label{singu}
(y+1)\ln(By)=-1.
\end{equation}
\item When $B<0$, the branches are
\begin{itemize}
\item $W^0_{\mathcal{L},<}(x) : [0, f(\delta_2)]\to [\delta_2,0]$ is strictly decreasing;
\item $W^1_{\mathcal{L},<}(x) : [f(\delta_1),f(\delta_2)]\to [\delta_1,\delta_2]$ is strictly increasing,
\item $W^2_{\mathcal{L},<}(x) : [f(\delta_1),0]\to (-\infty,\delta_1]$ is strictly decreasing,
\end{itemize}
where $\delta_1$ and $\delta_2$ are the two solutions to \eqref{singu} with $\delta_1<\frac{1}{B}<\delta_2<0$. 
\end{enumerate}
\end{thm}
\begin{proof} Consider the case when $B>0$. Let $x=f(y)=y\ln(By)e^y$. From equation \eqref{derivatives of log lambert},  the derivative of $y=W_\mathcal{L}(x)$ is not defined when $y$ satisfies \eqref{singu}. The solution $y=\delta$ to \eqref{singu} can be viewed as the intersection of the functions
$$g(y)=-1\;\;\;\;\mbox{and}\;\;\;\;h(y)=(y+1)\ln (By).$$
Clearly, the solution is unique. Thus, the derivative $\frac{dW_\mathcal{L}(x)}{dx}$ is not defined for $x=f(\delta)=\delta\ln (B\delta)e^{\delta}$. The value of $f(\delta)$ can then be used to determine the branches of $W_\mathcal{L}(x)$. To explicitly identify the said branches, the following information are important:
\begin{enumerate}
\item the value of $y$ must always be positive, otherwise, $\ln (By)$ is undefined;
\item the function $y=W_\mathcal{L}(x)$ has only one $y$-intercept, i.e., $y=\frac{1}{B}$;
\item if $y<\delta$, $(y+1)\ln (By)+1<0$ which gives $\frac{dy}{dx}<0$;
\item if $y>\delta$, $(y+1)\ln (By)+1>0$ which gives $\frac{dy}{dx}>0$;
\item if $y=\delta$, $(y+1)\ln (By)+1=0$ and $\frac{dy}{dx}$ does not exist
\end{enumerate}
These imply that 
\begin{enumerate}
\item when $y>\delta$, the function $y=W_\mathcal{L}(x)$ is increasing in the domain $(f(\delta),+\infty)$ with range $[\delta,+\infty)$ and the function crosses the $y$-axis only at $y=\frac{1}{B}$;
\item when $y<\delta$, the function $y=W_\mathcal{L}(x)$ is decreasing, the domain is $(f(\delta),0)$ and the range is $[\delta,0)$ because this part of the graph does not cross the $x$-axis and $y$-axis;
\item when $y=\delta$, the line tangent to the curve at the point $(f(\delta),\delta)$ is a vertical line.
\end{enumerate}
These proved the case when $B>0$. For the case $B<0$, the solution to \eqref{singu} can be viewed as the intersection of the functions
$$g(y)=-\frac{1}{y+1}\;\;\;\;\mbox{and}\;\;\;\;h(y)=\ln (By).$$
These graphs intersect at two points $\delta_1$ and $\delta_2$. Thus, the derivative $\frac{dW_\mathcal{L}(x)}{dx}$ is not defined for 
\begin{align*}
x_1&=f(\delta_1)=\delta_1\ln (B\delta_1)e^{\delta_1},\\
x_2&=f(\delta_2)=\delta_2\ln (B\delta_2)e^{\delta_2}.
\end{align*}
Note that
\begin{enumerate}
\item the value of $y$ must always be negative, otherwise, $\ln (By)$ is undefined;
\item the function $y=W_\mathcal{L}(x)$ has only one $y$-intercept, i.e., $y=\frac{1}{B}$;
\item $g(y)$ is not defined at $y=-1$.
\end{enumerate}
The desired branches are completely determined as follows:
\begin{enumerate}
\item If $\delta_2<y<0$, then $(y+1)\ln By +1<0$. This gives $\frac{dy}{dx}<0$. Thus, the function $y=W_\mathcal{L}(x)$ is a decreasing function with domain $[0,f(\delta_2)]$ with range $[\delta_2,0]$;
\item If $\delta_1\le y\le \delta_2$, then $(y+1)\ln By +1>0$. This gives $\frac{dy}{dx}>0$. Thus, the function $y=W_\mathcal{L}(x)$ is increasing function with domain $[f(\delta_1),f(\delta_2)]$ and range $[\delta_1,\delta_2]$;
\item If $-\infty<y<\delta_1$, then $(y+1)\ln By +1<0$. This gives $\frac{dy}{dx}<0$. Thus $y=W_\mathcal{L}(x)$ is a decreasing function with domain $[f(\delta_1),0]$ and range $[-\infty,\delta_1]$.
\end{enumerate}
These complete the proof of the theorem.
\end{proof}

Figure 1 depicts the graphs of logarithmic Lambert function (red color) when $B=1$ and $B=-1$. The $y$-coordinates of the points of intersection of the blue and gray colored graphs correspond to the value of $\delta, \delta_1$ and $\delta_2$. 

\begin{figure}[t!]
\centerline{\includegraphics[width=6cm]{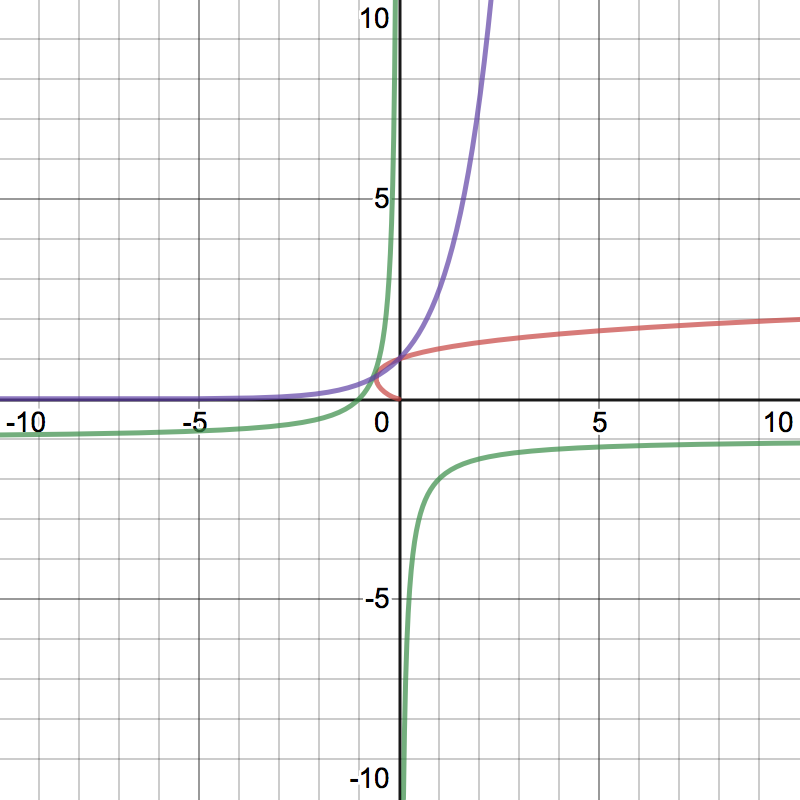}$\;\;$\includegraphics[width=6cm]{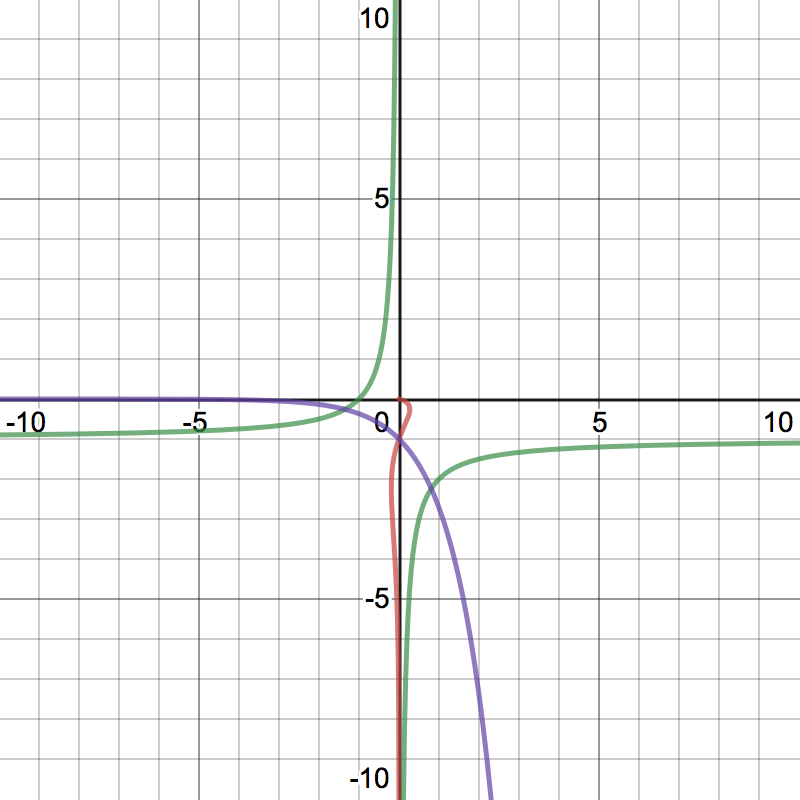}}
%\vspace*{1pt}
\centerline{\footnotesize{\textbf{Figure 1}}. {\footnotesize{Graphs of Logarithmic Lambert Function with $B = 1, -1$}}}
\centerline{{\footnotesize{The graphs with red, blue and gray colors are the graphs of }}}
\centerline{{\footnotesize{$x=f(y)$, $x=g(y)$ and $x=h(y)$, respectively.}}}
\end{figure}

\vspace{5pt}

\bigskip
\section{Applications to Classical Ideal Gas}
As discussed in \cite{Chandrashekar-Segar}, the microstate of a system of $N$ particles can be represented by a single point in the 2DN dimensional phase space. Corresponding to a particular value of the heat function which is a macrostate, is a huge number of microstates. The total number of microstates has to be computed in as it is a measure of entropy $S$. The points denoting the microstates of the system lie so close to each other that the surface area of the constant heat function curve in the phase space is regarded as a measure of the total number of microstates. 

\smallskip
In this section, applications of the logarithmic Lambert function to classical ideal gas in the four adiabatic ensembles are derived. In what follows, $m$ denotes the mass of the system; $P$, pressure; $V$, volume; $h$, Planck's constant (See \cite[p. 119]{Reif}); and $\mu$, the chemical potential of the system (see \cite{Thermostatics}).

\bigskip
\noindent {\bf Microcanonical Ensemble $(N,V,E)$}

\bigskip

\indent The Hamiltonian of a nonrelativistic classical ideal gas in $D$ dimensions is 
\begin{equation}
H=\sum_i\frac{p_i^2}{2m},\ \ p_i=|p_i|
\end{equation}
where $p_i$(for $i=1,2, \cdots N$) represent the $D$-dimensional momentum of the gas molecules. This classical nonrelativistic ideal gas is studied in the micro-canonical ensemble. In order to compute the entropy of the system, the phase space volume enclosed by the constant energy curve is computed and is given by (see \cite{Chandrashekar-Segar}),
\begin{equation}
\sum (N,V,E)=\frac{V^N}{N!} \frac{M^N}{\Gamma \left(\frac{DN}{2}+1 \right)}E^{\frac{DN}{2}},
\end{equation}
where
\begin{equation}\label{M}
M=\left( \frac{2\pi m}{h^2} \right)^\frac{D}{2}.
\end{equation}    
The three-parameter entropy of the system is 
$$S_{q,q',r}=k \ln_{q,q',r} \sum (N,V,E)\qquad\qquad\qquad\qquad\qquad\qquad\qquad\qquad\qquad\qquad\qquad\qquad$$
\begin{align*}
&=\frac{k}{1-r}\left[\exp\left(\frac{1-r}{1-q'}\exp \left(\frac{1-q'}{1-q} \left(\left(\sum (N,V,E) \right)^{1-q}-1 \right) \right)-1 \right)-1 \right].
\end{align*}
Computing the inner exponential,
$$\exp\left(\frac{1-q'}{1-q}\left( \left(\sum(N,V,E) \right) ^{1-q}-1\right) \right)\qquad\qquad\qquad\qquad\qquad\qquad\qquad$$
\begin{align*}
&=\exp\left(\frac{1-q'}{1-q}\left(\left( \frac{V^NM^NE^{\frac{DN}{2}}}{N!\Gamma \left(\frac{DN}{2}+1 \right)} \right)^{1-q} -1 \right) \right)\\
&=\exp\left(\frac{1-q'}{1-q}\left(\xi_{mc}^{1-q} \cdot E^{\frac{DN}{2}(1-q)}-1 \right) \right),
\end{align*}
where
\begin{equation*}
\xi_{mc}=\frac{V^NM^N}{N!\Gamma \left(\frac{DN}{2}+1 \right)}.
\end{equation*}
Let 
$$u=\frac{1-q'}{1-q}\xi_{mc}^{1-q} E^{\frac{DN}{2}(1-q)}.$$ 
Then
\begin{equation}
S_{q,q',r}=\frac{k}{1-r}\left[\exp\left(\frac{1-r}{1-q'} e^uA \right) e^{\frac{-(1-r)}{1-q'}}-1\right],
\end{equation}
where 
\begin{equation}\label{A}
A=e^{\frac{-(1-q')}{1-q'}}. 
\end{equation}
\noindent With 
\begin{equation}\label{z}
z=e^u, 
\end{equation}

\begin{equation}
S_{q,q',r}=\frac{k}{1-r} \left[ e^{\frac{1-r}{1-q'}zA}e^{\frac{-(1-r)}{1-q'}}-1\right].
\end{equation}
From the definition of temperature,
\begin{equation}
T=\left(\frac{\partial S_{q,q',r}}{\partial E} \right)^{-1}.
\end{equation}
Hence,
\begin{equation*}
\frac{1}{T}=\frac{\partial S_{q,q',r}}{\partial E}=\frac{k}{1-r} \left[e^{\frac{-(1-r)}{1-q'}} e^{\frac{1-r}{1-q'}zA}\left(\frac{1-r}{1-q'} \right)A \frac{dz}{du}\cdot \frac{du}{•dE}\right],
\end{equation*}
where
\begin{align}
\frac{dz}{du}&=e^u=z \nonumber \\
\frac{du}{dE}&=(1-q')\xi_{mc}^{1-q}\frac{DN}{2}E^{\frac{DN}{2}(1-q)-1}. 
\end{align}
Thus,
\begin{equation*}
\frac{1}{T}=k e^{\frac{-(1-r)}{1-q'}} e^{\frac{1-r}{1-q'}zA}zA\xi_{mc}^{1-q}\frac{DN}{2}E^{\frac{DN}{2}(1-q)-1}.
\end{equation*}
\noindent Substituting 
\begin{equation}\label{beta}
\beta=\frac{1}{kT}, 
\end{equation}
\noindent the preceding equation becomes
\begin{equation}
\label{largeN}
\frac{\beta \exp^{\frac{1-r}{1-q'}}}{\frac{DN}{2}A}=e^{\frac{1-r}{1-q'}e^uA}e^u\xi_{mc}^{1-q}E^{\frac{DN}{2}(1-q)-1}.
\end{equation}
Let 
\begin{equation}\label{y}
y=\frac{1-r}{1-q'}Ae^u. 
\end{equation}
Then
\begin{equation}
\label{EDN}
\frac{1-q}{\xi_{mc}^{1-q}(1-q')}\ln \left( \frac{y(1-q')}{(1-r)A}\right)=E^{\frac{DN}{2}(1-q)}.
\end{equation}
For large $N$, 
\begin{equation*}
\frac{DN}{2}(1-q)-1 \approx \frac{DN}{2}(1-q).
\end{equation*}
Then \eqref{largeN} is approximated by
\begin{equation}
\label{Bet}
\frac{2\beta e^{\frac{1-r}{1-q'}}(1-r)}{(1-q)DN}=ye^y \ln (By),
\end{equation}
where 
\begin{equation}\label{B}
B=\frac{1-q'}{(1-r)A}.
\end{equation}
Let
$$x=\frac{2\beta e^{\frac{1-r}{1-q'}}(1-r)}{(1-q)DN}.$$
It follows from \eqref{Bet} that 
$$y=W_\mathcal{L}(x).$$ 
\noindent From \eqref{EDN}, 
\begin{equation*}
E^{\frac{DN}{2}(1-q)}=\frac{1-q}{\xi_{mc}^{1-q}(1-q')}\ln \left(B W_\mathcal{L}\left(\frac{2\beta (1-r) e^{\frac{1-r}{1-q'}}}{(1-q)DN} \right) \right)
\end{equation*} 
which gives 
\begin{equation}\label{energy in Macrocanonical ensemble}
E=\left[\frac{1-q}{\xi_{mc}^{1-q}(1-q')}\ln \left(B W_\mathcal{L} \left(\frac{2\beta (1-r) e^{\frac{1-r}{1-q'}}}{(1-q)DN} \right) \right)\right]^{\frac{2}{DN (1-q)}}.
\end{equation}
The specific heat at constant volume is
$$C_V=\frac{\partial E}{\partial T}.$$
Let
\begin{equation}
a=\frac{1-q}{(1-q')\xi_{mc}^{1-q'}},\;\;\;\;\; c=\frac{2(1-r)e^{\frac{1-r}{1-q'}}}{(1-q)DN}.
\end{equation}
Then \eqref{energy in Macrocanonical ensemble} can be written
$$E=\left[a\ln(B W_\mathcal{L}(c/kT))\right]^{\frac{2}{(1-q)DN}}.$$
Taking the partial derivative of $E$ with respect to $T$,
\begin{equation}\label{C_V}
C_V=\frac{\partial E}{\partial T}=\frac{2a^{\frac{2}{DN(1-q)}}[\ln BW_\mathcal{L}(c/kT)]^{\frac{2}{DN(1-q)}-1}}{W_\mathcal{L}(c/kT)}\frac{d}{dT}W_\mathcal{L}(c/kT),
\end{equation}
where
\begin{equation}
\frac{d}{dT}W_\mathcal{L}(c/kT)=\frac{\frac{-c}{kT^2}e^{-W_L(c/kT)}}{[W_L(c/kT)+1]\ln BW_\mathcal{L}(c/kT)+1}.
\end{equation}

\bigskip 
Let us consider the following regions depending on the values of the deformation parameters $q$, $q'$, and $r$. (i) When $r>1$ and $q>1$, the argument of $W_\mathcal{L}$ is positive. If $q'>1$, then $B>0$ and $W_\mathcal{L}$ must be the principal branch $W_\mathcal{L}^0$. If $q'<1$, then $B<0$. With the argument of $W_\mathcal{L}$ being positive, we shall take $W_\mathcal{L}$ to be the branch $W_{\mathcal{L},<}^0$. (ii) If $r<1$ and $q<1$, then the argument of $W_\mathcal{L}$ is positive. If $q'<1$, then $B>0$ and we take $W_\mathcal{L}$ to be the principal branch $W_\mathcal{L}^0$. If $q'>1$, $B<0$ and again we take $W_\mathcal{L}$ to be the branch $W_{\mathcal{L},<}^0$. (iii) If $r<1$, $q>1$, then the argument of $W_\mathcal{L}$ is negative. If $q'<1$, then $B$ is positive and we take $W_\mathcal{L}$ to be the principal branch $W_\mathcal{L}^0$. If $q'>1$, then $B$ is negative. Here, we have two choices for $W_\mathcal{L}$, either $W_{\mathcal{L},<}^1$ or $W_{\mathcal{L},<}^2$. Since the heat function must be a continuous function of the deformation parameters we must in this case restrict $W_\mathcal{L}$ to $W_{\mathcal{L},<}^1$.

\bigskip
The specific heat at constant volume is either positive or negative depending on the values of the deformation parameters $q$, $q'$, $r$. Because of some resrictions, the results in \cite{Chandrashekar-Segar} corresponding to \eqref{energy in Macrocanonical ensemble} and \eqref{C_V} cannot be recovered even when $r\to 1$.

In the applications to the three other adiabatic ensembles, the definition of $M$, $A$, $z$, $\beta$, $y$ and $B$,   will be the same as that in \eqref{M}, \eqref{A}, \eqref{z}, \eqref{beta}, \eqref{y}, and \eqref{B}, respectively. However, since definition of $u$ differs in every ensemble, $z$ will have different values in every ensemble.

\bigskip
\noindent {\bf The Isoenthalpic-Isobaric Ensemble $(N, P, H)$}

\bigskip

A system which exchanges energy and volume with its surroundings in such a way that its enthalpy remains constant is described by the isoenthalpic-isobaric ensemble. To be able to compute the entropy of the system the following expression of phase space volume obtained in \cite{Chandrashekar-Segar} will be used,
\begin{equation}
\Sigma(N,P,H)=\mathcal{M}^N\left(\frac{1}{P}\right)^N\frac{H^{\frac{DN}{2}+N}}{\Gamma\left(\frac{DN}{2}+N+1\right)}.
\end{equation}  
The three-parameter entropy of the system is
\begin{align}
S_{q,q^{\prime},r} &=k\ln_{q, q^{\prime }, r}\Sigma(N,P,H)\nonumber\\
&=\frac{k}{1-r}\left[e^{\frac{1-r}{1-q^{\prime}}zA}\ \ e^{-\frac{\left(1-r\right)}{1-q^{\prime}}}-1\right],\label{Isoentalpic-isobaric entropy}
\end{align}
where $\alpha=\frac{DN}{2}+N$,
$$u=\frac{1-q^{\prime}}{1-q}\xi^{1-q}_{ie}H^{\alpha(1-q)},$$
$$\xi_{ie}=\frac{\mathcal{M}^N\left(\frac{1}{P}\right)^N}{\Gamma\left(\frac{DN}{2}+N+1\right)}.$$ 

From the definition of temperature,
\begin{equation}
\frac{1}{T}=\frac{\partial S_{q,q^{\prime},r}}{\partial H}=k\left[e^{-\frac{\left(1-r\right)}{1-q^{\prime}}} \ e^{\frac{1-r}{1-q^{\prime}}zA}\cdot (1-q^{\prime})A \ \frac{dz}{du}\frac{du}{dH} \right],\nonumber\\
\end{equation}
where
\begin{align}
\frac{dz}{du}&=e^u=z,\nonumber\\
\frac{du}{dH}&=\alpha\left(1-q^{\prime} \right)\xi^{1-q}_{ie}H^{\alpha(1-q)-1}.
\label{dv/dH}
\end{align}
For large $N$, approximate
\begin{equation*}
H^{\alpha(1-q)-1}\approx H^{\alpha(1-q)}.
\end{equation*}
Then
\begin{equation*}
\frac{du}{dH}\approx \alpha\left(1-q^{\prime} \right)\xi^{1-q}_{ie}H^{\alpha(1-q)}
\end{equation*}
and
\begin{equation}\label{1/T}
\frac{1}{T}
=k\left[e^{-\frac{\left(1-r\right)}{1-q^{\prime}}} \ e^{\frac{1-r}{1-q^{\prime}}zA}Az \ \alpha \ \xi^{1-q}_{ie}H^{\alpha(1-q)}\right].\\
\end{equation}

\noindent Substituting $\beta=\frac{1}{kT}$, \eqref{1/T} becomes
\begin{equation}\label{from 1/T}
\frac{\beta e^{\frac{1-r}{1-q^{\prime}}}}{A\alpha \xi^{1-q}_{ie}}=e^{\frac{1-r}{1-q^{\prime}}zA}z H^{\alpha(1-q)},
\end{equation}
from which  
\begin{equation}\label{First expression for H}
H^{\alpha(1-q)}=\frac{1-q}{\xi^{1-q}_{ie}\left(1-q^{\prime}\right)}\ln\left(\frac{y\left(1-q^{\prime}\right)}{(1-r)A}\right),
\end{equation}
where $y$ is defined in \eqref{y}. Then \eqref{from 1/T} can be written
\begin{equation}\label{eq7}
\frac{\beta e^{\frac{1-r}{1-q^{\prime}}}(1-r)}{\alpha(1-q)}=ye^y\ln(By).
\end{equation}
It follows from \eqref{eq7} that 
$$y=W_{\mathcal{L}}\left(\frac{\beta e^{\frac{1-r}{1-q^{\prime}}}(1-r)}{\alpha(1-q)}\right).$$
From \eqref{First expression for H},
\begin{equation}\label{Final expression for H}
H=\left[\frac{(1-q)}{\xi^{1-q}_{ie}\left(1-q^{\prime}\right)}\ln\left(\frac{1-q^{\prime}}{(1-r)A}W_{\mathcal{L}}\left(\frac{\beta e^{\frac{1-r}{1-q^{\prime}}}(1-r)}{\alpha(1-q)}\right)\right)\right]^{\frac{1}{\alpha(1-q)}}.
\end{equation}
Let 
$$a=\frac{1-q}{\xi^{1-q}_{ie}\left(1-q^{\prime}\right)}, \;\;\;\; c=\frac{ e^{\frac{1-r}{1-q^{\prime}}}(1-r)}{\alpha(1-q)}.$$
Then
\begin{equation}\label{Expression for H with a,c terms}
H=\left[a\ln\left(B\ W_{\mathcal{L}}\left(\frac{c}{kT}\right)\right)\right]^{\frac{1}{\alpha(1-q)}}.
\end{equation}
The specific heat at constant pressure is
\begin{equation*}
C_P=\frac{\partial H}{\partial T}.
\end{equation*}
Taking the partial derivative of \eqref{Expression for H with a,c terms},
\begin{align*}
\frac{\partial H}{\partial T}&=\frac{1}{\alpha(1-q)}\left[a\ln\left(B\ W_{\mathcal{L}}\left(\frac{c}{kT}\right)\right)\right]^{\frac{1}{\alpha(1-q)}-1}\cdot \frac{a \ \frac{d}{dT}W_{\mathcal{L}}\left(\frac{c}{kT}\right)}{\left(\ W_{\mathcal{L}}\left(\frac{c}{kT}\right)\right)}
\end{align*}
Using the derivative of $W_{\mathcal{L}}(x)$ with respect to $x$ given in Section 2,
\begin{equation}
C_P=\frac{\frac{-c}{\alpha(1-q)kT^2}a^{\frac{1}{\alpha(1-q)}}e^{-W_L(c/kT)}[\ln bW_L(c/kT)]^{\frac{1}{\alpha(1-q)}-1}}{W_L(c/kT)\{[W_L(c/kT)+1]\ln BW_L(c/kT) +1\}}.
\end{equation}

The specific heat at constant pressure is either positive or negative depending on the values of the deformation parameters $q$, $q'$ and $r$. Corresponding results in \cite{Chandrashekar-Segar} cannot be recovered even when $r\to 1$ as we have used an approximation for an expression for $H$ in the computation and due to restrictions in the argument of $W_\mathcal{L}$.

\bigskip
\noindent {\bf The $(\mu,V,L)$ Ensemble}

\bigskip

\indent The Hill energy L is the heat function corresponding to the ($\mu$,V,L) ensemble. For large N, an approximate expression of the phase space volume obtained in \cite{Chandrashekar-Segar} is 
\begin{equation}
\sum{(\mu,V,L)}=\exp{\left(\frac{D}{2}\frac{L}{\mu}\right)}\exp{\left(\frac{V{\mu}^{\frac{D}{2}}e^{\frac{D}{2}}M}{(\frac{D}{2})^\frac{D}{2}}\right)}.
\label{sum of uvl}
\end{equation}
This is a first order approximation of the exact value
\begin{equation*}
\sum{(\mu,V,L)}=\sum_{N=0}^{\infty} \frac{V^{N}}{N!}\frac{M^N}{\Gamma(\frac{DN}{2}+1)}(L+\mu N)^{\frac{DN}{2}}.
\end{equation*}
The 3-parameter entropy of the classical ideal gas in this adiabatic ensemble is 

\begin{equation}
S_{q,q',r}
=\frac{k}{1-r}\left[\exp\left(\frac{1-r}{1-q'}z A\right)e^{-\frac{1-r}{1-q'}}-1\right]
\label{expression for 3-parameter entropy (uvl)}
\end{equation}

where $z=e^u$,
\begin{equation}
\xi_{he}=\exp\left(\frac{V(\mu e)^{\frac{D}{2}}M}{(\frac{D}{2})^{\frac{D}{2}}}\right),
\end{equation}
\begin{equation}
u=\frac{1-q'}{1-q}\xi^{1-q}_{he}e^{\frac{DL}{2\mu}(1-q)}.
\end{equation}

From the definition of temperature,
\begin{align}\label{temperature uvl}
\frac{1}{T}
&=\frac{\partial S_{q,q',r}}{\partial L} \nonumber\\
&=k\left[e^{-\frac{1-r}{1-q'}}e^{\frac{1-r}{1-q'}zA}\cdot\frac{A}{1-q'}A\cdot\frac{dz}{du}\cdot\frac{du}{dL}\right],
\end{align}
where
\begin{equation}
\frac{du}{dL}=\frac{(1-q')D}{2\mu} \ \xi^{1-q}_{he} \ e^{\frac{DL}{2\mu}(1-q)}.
\end{equation}
Then \eqref{temperature uvl} becomes 
\begin{equation}
\label{from the defn of temperature (uvl)}
\frac{2\beta\mu\exp\left(\frac{1-r}{1-q'}\right)}{AD\xi_{he}^{1-q}}=\exp\left(\frac{1-r}{1-q'}zA\right)z e^{\frac{DL}{2\mu}(1-q)},
\end{equation}

\noindent and further 
\begin{equation*}
\frac{1-q}{(1-q')\xi^{1-q}_{he} } \ \ln\left(\frac{y(1-q')}{(1-r)A}\right)=e^{\frac{DL}{2\mu}(1-q)}, 
\end{equation*}
from which
\begin{equation}
L=\frac{2\mu}{D(1-q)} \ \ln\left(\frac{1-q}{(1-q')\xi^{1-q}_{he} } \ \ln\left(\frac{y(1-q')}{(1-r)A}\right)\right).
\label{expression for L}
\end{equation}
To solve for $y$, rewrite \eqref{from the defn of temperature (uvl)} into
\begin{equation}
\frac{\beta(1-r)e^{\frac{1-r}{1-q'}}}{\frac{D}{2\mu}(1-q)}=ye^y\ln\left(\frac{y(1-q')}{(1-r)A}\right).
\end{equation}
Thus,
\begin{equation}
y=W_\mathcal{L}\left(\frac{2\beta\mu \ e^{\frac{1-r}{1-q'}} \ (1-r)}{D(1-q)}\right).
\label{W_L in uvl}
\end{equation}
Therefore
\begin{align}
L&=\frac{2\mu}{D(1-q)} \times \nonumber \\
&\;\;\;\;\;\ln\left(\frac{1-q}{(1-q')\xi^{1-q}_{he} } \ \ln\left(B \ W_\mathcal{L}\left(\frac{2\beta\mu \ e^{\frac{1-r}{1-q'}} \ (1-r)}{D(1-q)}\right)\right)\right).
\label{L}
\end{align}
The specific heat at constant volume is
\begin{equation*}
C_V=\frac{\partial L}{\partial T}.
\end{equation*}
Let 
$$a=\frac{1-q}{(1-q')\xi^{1-q}_{he}},\;\;\;\; c=\frac{2\mu (1-r) \ e^{\frac{1-r}{1-q'}}}{D(1-q)}.$$
Then
\begin{align*}
L&=\frac{2\mu}{D(1-q)} \ \ln\left(a \ \ln\left(B \ W_\mathcal{L}\left(c\beta\right)\right)\right), \\
\frac{\partial L}{\partial T}&=\frac{2\mu}{D(1-q)}\cdot \frac{\frac{d}{dT}\ \ln\left(B \ W_\mathcal{L}\left(c\beta\right)\right)}{\ln\left(B \ W_\mathcal{L}\left(c\beta\right)\right)}, 
\end{align*}
where
\begin{equation}
\frac{d}{dT} \ \ln\left(B \ W_\mathcal{L}\left(c\beta\right)\right) 
=\frac{-c}{kT^2}e^{-W_\mathcal{L}(c\beta)}{W_\mathcal{L}\left(c\beta\right)\left\{[W_\mathcal{L}(c\beta)+1]\ln\beta W_\mathcal{L}(c\beta)\;+1\right\}}
\end{equation}
Thus,
\begin{equation}
\frac{\partial L}{\partial T}
=\frac{2\mu}{D(1-q)}\;\frac{(-c/kT^2) e^{-W_\mathcal{L}(c/kT)}[ln (BW_\mathcal{L}(c/kT))]^{-1}}{W_\mathcal{L}(c/kT)\left([W_\mathcal{L}(c/kT)+1]\ln (BW_\mathcal{L}(c/kT)) +1\right)}.
\label{CV}
\end{equation}

The specific heat at constant volume is either positive or negative depending on the values of the deformation parameters $q$, $q'$ and $r$. Corresponding results in \cite{Chandrashekar-Segar} cannot be recovered even for $r\to 1$ due to restrictions in the argument of $W_\mathcal{L}$.

\bigskip
\noindent {\bf The $(\mu,P,R)$ Ensemble}

\bigskip

\indent The adiabatic ensemble with both the number and volume fluctuations is illustrated here using the classical ideal gas. \\
\indent In the large N limit, the expression of the phase space volume is (see \cite{Chandrashekar-Segar})
\begin{equation}
\sum(\mu,P,R)=\exp\left(\frac{DR}{\mu}\right)\left[1-\frac{M}{P}\left(\frac{\mu}{\alpha}\right)^{\alpha}\exp\alpha\right]^{-1},
\label{sum of upr}
\end{equation}
where $\alpha=\frac{DN}{2}+N$. \\
The three-parameter entropy is 

\begin{equation}
S_{q,q',r}=\frac{k}{1-r}\left[e^{\frac{1-r}{1-q'}zA}\cdot e^{-\frac{1-r}{1-q'}}-1\right], 
\label{expression for 3-parameter entropy (upr)}
\end{equation}
where  $z=e^u$,  
\begin{equation} 
\xi_{re}=\left(1-\frac{M}{P}\left(\frac{\mu}{\alpha}\right)^{\alpha}e^{\alpha}\right)^{-1},
\end{equation} 

\begin{equation}
u=\frac{1-q'}{1-q} \ \xi_{re}^{1-q} \ e^{\frac{DR}{\mu}(1-q)}.
\end{equation}

From the definition of temperature, 
\begin{align*}
\frac{1}{T}
&=\frac{\partial S_{q,q',r}}{\partial R} \\
&=k\left[e^{-\frac{1-r}{1-q'}}e^{\frac{1-r}{1-q'}zA}\cdot\frac{A}{1-q'}\cdot\frac{dz}{du}\cdot\frac{du}{dR}\right].
\end{align*}
where
\begin{equation}
\frac{du}{dR}=\frac{(1-q')D}{\mu} \ \xi^{1-q}_{re} \ e^{\frac{DR}{\mu}(1-q)}.
\end{equation}
Then 
\begin{equation}
\frac{1}{T}=k\left[e^{-\frac{1-r}{1-q'}}e^{\frac{1-r}{1-q'}zA}Az \ \frac{D}{\mu} \ \xi^{1-q}_{re} \ e^{\frac{DR}{\mu}(1-q)}\right], \nonumber 
\end{equation}
from which
\begin{equation}
\frac{\mu \beta \ e^{\frac{1-r}{1-q'}}}{AD \ \xi^{1-q}_{re}}=e^{\frac{1-r}{1-q'}zA} \ z \ e^{\frac{DR}{\mu}(1-q)}. 
\label{from the defn of temperature (upr)}
\end{equation}

\noindent Then
\begin{equation}
R=\frac{\mu}{D(1-q)} \ \ln\left(\frac{1-q}{(1-q')\xi^{1-q}_{re} } \ \ln\left(\frac{y(1-q')}{(1-r)A}\right)\right).
\label{expression for R}
\end{equation}

\noindent From \eqref{from the defn of temperature (upr)},
\begin{equation}
\frac{\mu \beta \ (1-r) \ e^{\frac{1-r}{1-q'}}}{D(1-q)}=ye^y \ln By.
\end{equation}
Thus 
\begin{equation*}
y=W_\mathcal{L}\left(\frac{\mu\beta(1-r)e^{\frac{1-r}{1-q'}}}{D(1-q)}\right).
\end{equation*}
Consequently, from \eqref{expression for R}, the Ray energy of the system is
\begin{align}
R&=\frac{\mu}{D(1-q)}\times\nonumber\\
&\;\;\;\ln\left(\frac{1-q}{(1-q')\xi^{1-q}_{re} } \ \ln\left(\frac{(1-q')}{(1-r)A} \ W_\mathcal{L}\left(\frac{\mu\beta(1-r)e^{\frac{1-r}{1-q'}}}{D(1-q)}\right)\right)\right).
\label{R}
\end{align}

\noindent The specific heat at constant pressure is
\begin{equation*}
C_P=\frac{\partial R}{\partial T}.
\end{equation*}
Let 
$$a=\frac{1-q}{(1-q')\xi^{1-q}_{re}},\;\;\;\; c=\frac{\mu (1-r) \ e^{\frac{1-r}{1-q'}}}{D(1-q)}.$$
Then
\begin{equation}
R=\frac{\mu}{D(1-q)} \ \ln\left(a \ \ln\left(B \ W_\mathcal{L}\left(c/kT\right)\right)\right)
\end{equation}
and
\begin{align}
C_P&=\frac{\mu}{D(1-q)}\times\nonumber\\
&\;\;\; \left[\frac{\frac{-c}{kT^2} \ e^{-W_L\left(\frac{c}{kT}\right)}\left[\ln\left(B \ W_L\left(\frac{c}{kT}\right)\right)\right]^{-1}}{W_L\left(\frac{c}{kT}\right) \ \left\{\ln\left(B W_L\left(\frac{c}{kT}\right)\right) \ W_L\left(\frac{c}{kT}\right)+\ln\left(B W_L\left(\frac{c}{kT}\right)\right)+1\right\}}\right].
\label{CP}
\end{align}

The specific heat at constant pressure is either positive or negative depending on the values of the deformation parameters $q$, $q'$ and $r$. Corresponding results in \cite{Chandrashekar-Segar} cannot be recovered even for $r\to 1$ due to restrictions in the argument of $W_\mathcal{L}$.

\section{Conclusion}
We have investigated the adiabatic class of ensembles in the framework of generalized mechanics based on the three-parameter entropy.The derivative and branches of the function were in particular, useful in the applications to the thermostatics of the nonrelativistic ideal gas.  In the microcanonical ensemble and isoenthalpic-isobaric ensemble, the formulas for the three-parameter entropy for  an arbitrary number of particles were obtained. In the large $N$ limit the heat functions were obtained in terms of the temperature and expressed in terms of the logarithmic Lambert function. From the heat functions, the specific heats at constant temperature were computed. In the $(\mu,V,L)$ and the $(\mu, P, R)$  ensembles an approximate phase volume in the large $N$ limit was used and the three-parameter entropies of the ensembles were computed. From the entropy function the heat function and the specific heat were found and expressed also in terms of the logarithmic Lambert function. Applications of the $q$-entropy, $(q,q')$-entropy and the $(q,q',r)$-entropy to Maximum Entropy Theory in ecology is also an interesting research topic that one may pursue.

\section*{Acknowledgment}
This research is funded by Cebu Normal University (CNU) and the Philippine Commission  on Higher Education - Grants-in-Aid (CHED-GIA) for Research.

\section*{Data Availability Statement}
The computer programs and articles used to generate the graphs and support the findings of this study are available from the corresponding author upon request.

\end{document}